\documentclass[12pt, twosided]{article}
\usepackage{latexsym,enumerate}
\usepackage{amsmath}
\usepackage{amsthm,amsopn}
\usepackage{amstext,amscd,amsfonts,amssymb,fontenc,bbm}
\usepackage[ansinew]{inputenc}
\usepackage{verbatim}
\usepackage{graphicx}
\usepackage{epsfig}

\usepackage{pstricks,pst-node}
\usepackage{tikz}
\usepackage{pgf,tikz}
   \usepackage{epsfig}

\usepackage{lineno} 
\usepackage{textcomp}
\usepackage{mathtools} 

\newtheorem{theorem}{Theorem}[section]
\newtheorem{corollary}[theorem]{Corollary}
\newtheorem{claim}[theorem]{Claim}

\newtheorem{conjecture}[theorem]{Conjecture}
\newtheorem{proposition}[theorem]{Proposition}

\newtheorem{lemma}[theorem]{Lemma}
\newtheorem{observation}[theorem]{Observation}

\DeclarePairedDelimiter\floor{\lfloor}{\rfloor}
\newcommand\restrict[1]{\raisebox{-.5ex}{$|$}_{#1}} 
\def\zz{\mathbb{Z}}
\newcommand{\aut}{{\rm Aut}}

\usepackage{lineno}
\begin{document}

\title{The automorphism groups of some token graphs}

\author{Sof\'ia Ibarra 
\and Luis Manuel Rivera
}
\date{}

\maketitle{}

\begin{abstract}
 
The token graphs of graphs have been studied at least from the 80's with different names and by different authors. The Johnson graph $J(n, k)$ is isomorphic to the $k$-token graph of the complete graph $K_n$. To our knowledge, the unique results about the automorphism groups of token graphs are for the case of the Johnson graphs. In this paper we begin the study of the automorphism groups of token graphs of another graphs. In particular we obtain the automorphism group of the $k$-token graph of the path graph $P_n$, for $n\neq 2k$. Also, we obtain the automorphism group of the $2$-token graph of the following graphs: cycle, star, fan and wheel graphs. 
\end{abstract}


\section{Introduction}

Let $\Gamma$ be a simple graph of order $n$. Let $1\leq k \leq n-1$ be an integer. The {\it $k$-token graph}  $F_k(\Gamma)$ of $\Gamma$ is defined as the graph with vertex set all $k$-subsets of $V(\Gamma)$, where two vertices are adjacent in $F_k(\Gamma)$ whenever their symmetric difference is an edge of $\Gamma$. If $k \in \{1, n-1\}$, then $F_k(\Gamma)$ is isomorphic to $\Gamma$ and in this case we say that $F_k(\Gamma)$ is a {\it trivial token graph} of $\Gamma$. In fact, if $\Gamma$ is a graph of order $n$, then $F_k(\Gamma) \simeq F_{n-k}(\Gamma)$.  

The token graphs have been redefined several times and with different names. When $k=2$, this class of graphs are called {\it double vertex graphs} that were widely studied by Alavi et al.~\cite{alavi2,alavi1,alavi3, alavi4} and are the same that the  $2$-subgraph graphs defined in a thesis of G. Johns~\cite{johns}. In the work of Zhu et al. \cite{zhu}, the $k$-token graphs are named $n$-tuple vertex graphs. 
Later, T. Rudolph \cite{rudolph} redefined the double vertex graphs with the name of symmetric powers of graphs with the idea to study the graph isomorphism problem and  some problems in quantum mechanics.  There are several papers related with Rudolph's work, see, e. g., \cite{alzaga, aude, barghi, fisch} and the references therein. Some of them motivated by the connection between token graphs and the Heisenberg Hamiltonian (see, e. g.,  \cite{ouyang} and the references therein), that is related with the Heisenberg model \cite{heisen}, a quantum theory of ferromagnetism. 

Finally, R. Fabila-Monroy, et. al.~\cite{ruy}, in an independent way, reintroduce this concept but now with the name of token graphs and began a systematic study of  several combinatorial properties of this graphs: connectivity, diameter, cliques, chromatic number and Hamiltonian paths. In the last years, several groups of authors have continued with this line of research (see, e.g., \cite{alba, deepa, deepa2, gomez, paloma, leatrujillo, rive-tru}). For example, Carballosa et al.~\cite{carba} studied the planarity and regularity of token graphs and  Lea\~nos and Trujillo-Negrete \cite{leatrujillo} proved a conjecture of Fabila-Monroy, et. al~\cite{ruy} about the connectivity of token graphs. Finally, G\'omez Soto et al. \cite{gomez} found the packing number of the $2$-token graph of the path graph, that is equal to the size of largest binary code of length $n$ and constant weight $2$ that can correct a single adjacent transposition (sequence A085680 in~\cite{oeis}).

 When $\Gamma$ is the complete graph $K_n$, the $k$-token graph $F_k(\Gamma)$ is isomorphic to the Johnson graph $J(n, k)$ \cite{Jo}. To the knowledge of the authors, the only results about the automorphism groups of token graphs are about Johnson graphs. It is known that if $n \neq 2k$, then $\aut(F_k(K_n))\simeq S_n$ and if $n = 2k$, then $\aut(F_k(K_n))\simeq S_2 \times S_n$, where $S_n$ denotes the symmetric group on $n$ symbols  (see., e.g. \cite{jones, miraf, ramras}). 

In this work, we study the automorphism group of other token graphs. Our main results can be stated as two theorems:
\begin{theorem}\label{elte2}
Let $n \neq 4$ be an integer greater than $2$. If $\Gamma \in \{C_n, K_{1, n-1}, A_{1, n-1}, W_{1, n-1}\}$, 
then \[
\aut(F_2(\Gamma))=\aut(\Gamma),
\] 
where $C_n, K_{1, n-1}, A_{1, n-1}$ and $W_{1, n-1}$, denotes the cycle, star, fan and wheel graphs, respectively. 
\end{theorem}

\begin{theorem}\label{auto-path}
Let $P_n$ be the path graph of order $n>2$, with $n \neq 2k$. Then 
\[
\aut(F_k(P_n))=\aut(P_n).
\]
\end{theorem}

Theorem~\ref{elte2} is not true in general. For example, the automorphism group of the grid graph  $G_{2, 3}$ is of order $4$ but $|\aut(F_2(G_{2, 3}))|=8$. 

In the proofs of our results, we use elementary group theory, as in \cite{miraf2,miraf3,miraf}, and properties of token graphs. For the case of the token graphs of path graphs we obtain a formula for the distance between pair of vertices in $F_k(P_n)$ that generalizes the one given by Beaula et al.~\cite{beaula}.

The outline of this paper is as follows. In Section~\ref{sec1} we present some definitions, notation and some preliminary results. We show that $\aut(\Gamma)$ is a subgroup of $\aut(F_k(\Gamma))$, for every graph $\Gamma$. Also, we show that if $n=2k$, then $|\aut(F_k(\Gamma))|\geq 2|\aut(\Gamma)|$. The proof of Theorem~\ref{elte2} is worked for each case separately.  In Section~\ref{sec2} we show that $\aut(F_2(C_n))=\aut(C_n))$, for $n \neq 4$. In Section~\ref{sec3} we prove that if $\Gamma \in \{K_{1, n}, A_{1, n}, W_{1, n}\}$, then $\aut(F_2(\Gamma))=\aut(\Gamma)$, for $n \neq 4$. In Section~\ref{sec4} we prove Theorem~\ref{auto-path}. 

\section{Preliminaries and first results}\label{sec1}
In this paper, all our graphs are simple and finite, that is, a {\it graph} $\Gamma$ is a pair $(V(\Gamma), E(\Gamma))$ where $V(\Gamma)$ is a finite set and $E(\Gamma)$ is a subset of the set of all $2$-subsets of $V(\Gamma)$. An edge of a graph $\Gamma$ will be denoted by $\{u, v\}$ or $uv$, for $u, v \in V(\Gamma)$. We use $u \sim v$ to indicate that $u$ and $v$ are adjacent vertices, that is $uv \in E(\Gamma)$.  The {\it neighborhood} of a vertex $v$ is defined as $N(v)=\{u \in V(\Gamma) \colon uv \in E(\Gamma)\}$ and the {\it degree} $d(v)$ of $v$ is defined as $|N(v)|$. The neighborhood of a set of  vertices $X$ is defined as $N(X)=\bigcup_{x \in X}N(x) \setminus X$. Let $U$ be a subset of $V(\Gamma)$, we will use $\Gamma \langle U \rangle$ to denote the subgraph of $\Gamma$ induced by $U$. The graph difference $\Gamma - U$ is defined as the graph $\Gamma \langle  V(\Gamma) \setminus U \rangle$. 

Let $\Gamma_1$ and $\Gamma_2$ be two simple graphs. An {\it isomorphism} of $\Gamma_1$ onto $\Gamma_2$ is a bijection $\phi\colon V(\Gamma_1) \to V(\Gamma_2)$ such that $uv \in E(\Gamma_1)$ if and only if $\phi(u)\phi(v) \in E(\Gamma_2)$. An automorphism of a graph $\Gamma$ is an isomorphism of $\Gamma$ onto itself. The set of all automorphism of a graph $\Gamma$ is a subgroup of ${\rm Sym}(V(\Gamma))$, the group of all permutations of $V(\Gamma)$, an is denoted by $\aut(\Gamma)$. To obtain the automorphism group of graphs in general is a difficult problem. But it is possible to obtain this group for particular cases. 

It is well-known that $\aut(\Gamma)$ acts on $V(\Gamma)$. Let $v \in V(\Gamma)$, the {\it orbit} of $v$ is defined as $O(v)=\{f(v) \colon f \in \aut(\Gamma)\}$ and the {\it stabilizer} of $v$ is ${\rm Stab}(x)=\{f \in \aut(\Gamma) \colon f(x)=x\}$.  The orbit-stabilizer theorem says that $|\aut(\Gamma)|=|O(v)||{\rm Stab}(v)|$, for every $v \in  V(\Gamma)$.  Let $f \in \aut(\Gamma)$, we say that a vertex $x \in V(\Gamma)$ is a {\it fixed point} of $f$ if $f(x)=x$. We use ${\rm Fix}(f)$ to denote the set of fixed points of $f$. As usual, sometimes we write $\aut(\Gamma)=G$ instead of $\aut(\Gamma) \simeq G$.
We use $S_n$ to denote the symmetric group over $\{1, \dots, n\}$.   

The following observation (see \cite{ruy}) will be used, sometimes without reference, when we compute the degree of vertices in $F_k(\Gamma)$. 
\begin{observation}\label{grado}
The degree of a vertex $A$ in $F_k(\Gamma)$ is equal to the number of edges between $A$ and $V(\Gamma) \setminus A$.
\end{observation}
In particular if $\{x, y\} \in V(F_2(\Gamma))$, then  $d(\{x, y\})=d(x)+d(y)$ if $x \not \sim y$ and $d(\{x, y\})=d(x)+d(y)-2$ if $x \sim y$.

The following result, that appears in \cite{carba} and \cite{ruy}, will be useful in some of the proofs. 
\begin{proposition}\label{pborrado}
Let $X$ be a subset of $V(\Gamma)$ and $\Gamma'=\Gamma-X$. Then $F_k(\Gamma')$ is isomorphic to the graph obtained from $F_k(\Gamma)$ by deleting the vertices $A$ in $F_k(\Gamma)$ such that $A$ has al least one element of $X$.  
\end{proposition}

\subsection{First results}
Our first result shows an important relation between  $\aut(\Gamma)$ and $\aut(F_k(\Gamma))$. The proof of the following theorem is straightforward. 
\begin{theorem}\label{aut-sub}
Let $\Gamma$ be a graph. Then $\aut(\Gamma)$ is isomorphic to a subgroup of $\aut(F_k(\Gamma))$. In fact, if $\theta \in \aut(\Gamma)$, then the function $f_\theta \colon V(F_k(\Gamma)) \to V(F_k(\Gamma))$ defined as \[f_\theta(\{v_1, \dots, v_k\})=\{\theta(v_1), \dots, \theta(v_k)\}\] is an automorphism of $F_k(\Gamma)$.  
\end{theorem}

 The automorphism $f_\theta$ of $F_k(\Gamma)$, for $\theta \in \aut(\Gamma)$, defined in previous theorem is called the {\it automorphism induced } by $\theta$. When the context is clear, we use $\aut(\Gamma)$ as the set of automorphism of $\Gamma$ or as the subgroup of $\aut(F_k(\Gamma))$ induced by the automorphisms of $\Gamma$. We write $\aut(F_k(\Gamma))=\aut(\Gamma)$ to mean that every automorphism of $F_k(\Gamma)$ is induced by some  automorphism of $\Gamma$.   

Now, the following proposition shows that for $n=2k$, $\aut(F_k(\Gamma))$ has always more elements than $\aut(\Gamma)$.

\begin{theorem}
Let $\Gamma$ be a graph of order $n$, with $n$ even. The function $f_c \colon V(F_{n/2}(\Gamma)) \to V(F_{n/2}(\Gamma))$ defined as $f_c(A)=A^c$ is an automorphism, where $A^c=V(\Gamma) \setminus A$. Even more $f_c$ is not an induced automorphism of any $\phi \in \aut(\Gamma)$.
\end{theorem}
\begin{proof}
The proof is exactly the same that the given in the proof of Theorem 3.5 in \cite{miraf} for the case when $\Gamma$ is the complete graph $K_n$ (the graph $F_k(K_n)$ is isomorphic to the Johnson graph $J(n, k)$). 
\end{proof}
Clearly, the function $f_c$ in previous theorem is a fixed point free involution. 
\begin{corollary}
Let $n\geq 4$ be an even integer. If $\Gamma$ is a graph of order $n$, then 
\[
|\aut(F_{n/2}(\Gamma))|\geq 2|\aut(\Gamma)|.
\]

\end{corollary}

\section{Automorphism group of the $2$-token graph of cycle graphs}\label{sec2}

In this section we prove that $\aut(F_2(C_n))=\aut(C_n)$, for $n \neq 4$. In Figure~\ref{fig1} we show $F_2(C_7)$. Let $D_{2n}$ denote the dihedral group of $2n$ elements. It is well-known that $\aut(C_n)=D_{2n}$. Using computer software,  we obtain that $\aut(F_2(C_4))=S_2 \times D_8$. First we present some observations and results that will be useful. In this section, $V(C_n)=\{1, 2 \dots, n\}$ and $E(C_n)=\{\{i, i+1\} \colon 1\leq i \leq n-1\} \cup \{\{1, n\}\}$.

\begin{observation}\label{obs-vecinos} Let $n \geq 4$ be an integer.
 \begin{enumerate}
\item If $v \in F_2(C_n)$, then $d(v) \in \{2, 4\}$.
\item $|N(u)\cap N(v)|\leq 2$, for every pair of vertices $u, v \in F_2(C_n)$.
\end{enumerate}
\end{observation}

We use $i\oplus j$ and $i \ominus j$ to denote the sum $(i+j) \bmod n$ and $(i-j) \bmod n$, respectively, with the convention that $n \equiv n \pmod n$.
Let $r=\floor{n/2}$. We define the following subsets of $V(F_2(C_n))$. 
\[
L_q=\{\{i, i \oplus q \} \colon 1\leq i \leq n\},
\]
where $1\leq q \leq r$.

 \begin{figure}
\begin{center}
\includegraphics[scale=0.6]{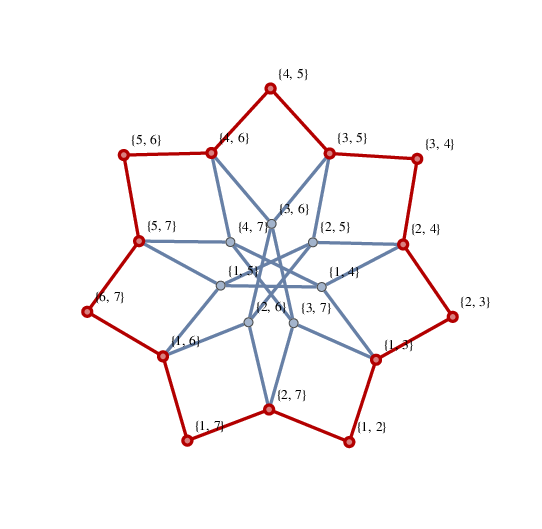}
\caption{The $2$-token graph of $C_7$, where $V(C_7)=\{1, \dots, 7\}$. The red subgraph is induced by $L_1\cup L_2$.} 
\label{fig1}
\end{center}
\end{figure}
The proof of the following proposition is an easy exercise. 
\begin{proposition}\label{properties2Cn} Let $n\geq 3$ be an integer and $r=\floor{n/2}$.
\begin{enumerate}  
\item \label{properties2Cn1} If $n$ is even, then $|L_{n/2}|=n/2$ and $|L_q|=n$, for $1\leq q <r$.
\item \label{properties2Cn2} If $n$ is odd, then $|L_q|=n$, for $1\leq q \leq r$.
\item \label{properties2Cn3} The set $L=\{L_1, \dots, L_r\}$ is a partition of $V(F_2(C_n))$.
\item \label{properties2Cn4} Let $n \geq 6$ and $3 \leq q \leq r$. If $\{i, i\oplus q\} \in L_{q}$, with $1\leq i \leq n$, then two neighbors, say $B, C$, of $\{i, i\oplus q\}$ belongs to $L_{q-1}$ and the vertex in $N(B) \cap N(C) \setminus\{\{i, i\oplus q\}\}$ belongs to $L_{q-2}$.
\item \label{properties2Cn5} Let $v \in V(F_2(C_n))$, then $d(v)=2$ if and only if $v \in L_1$. 
\end{enumerate}
\end{proposition}
In Figure~\ref{fig1} we show the subgraph of $F_2(C_7)$ induced by $L_1\cup L_2$.

\begin{figure}
\begin{center}
\includegraphics[scale=0.45]{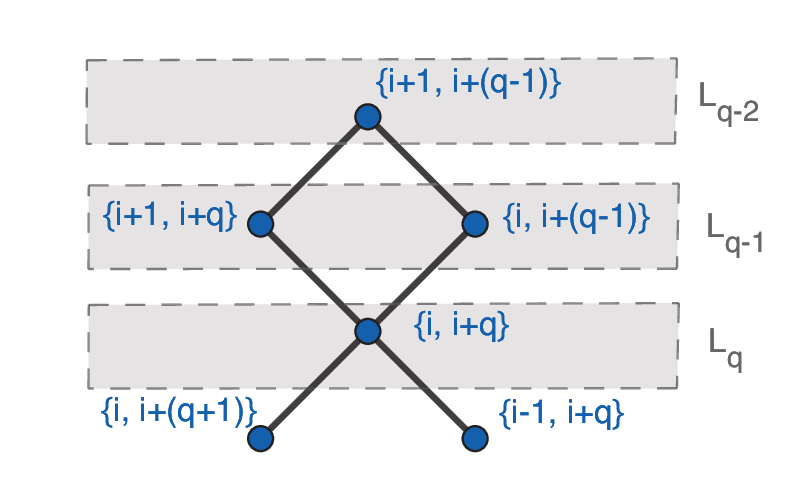}
\caption{An illustration of Proposition~\ref{properties2Cn}(\ref{properties2Cn4}).}
\label{fig2}
\end{center}
\end{figure}

\begin{proposition}\label{iso-cycle}
Let $n\geq 6$. The subgraph of $F_2(C_n)$ induced by $L_1 \cup L_2$ is isomorphic to $C_{2n}$.
\end{proposition}
\begin{proof}

First note that if $\{i, i\oplus1\}$ in $L_1$, then
\[
N(\{i, i\oplus1\})=\{\{i\ominus1, i\oplus1\}, \{i, i\oplus2\}\}, 
\]
and if $\{i, i\oplus2\}$ in $L_2$, then

 \[
 N(\{i, i\oplus2\})=\{\{i, i\oplus1\}, \{i, i\oplus3\}, \{i\oplus1, i\oplus2\}, \{i\ominus1, i\oplus2\}\}.
 \]
Since $n\geq 6$, then $N(\{i, i\oplus1\}) \subset L_2$, $N(\{i, i\oplus2\})\cap L_1=\{\{i, i\oplus1\}, \{i\oplus1, i\oplus2\}\}$ and $N(\{i, i\oplus2\})\cap L_3=\{\{i, i\oplus3\}, \{i\ominus1, i\oplus2\}\}$. Thus, it is easy to check that the function $\phi \colon V( L_1 \cup L_2) \to V(C_{2n})$ given by $\phi\left(\{i, i\oplus1\}\right)=2i-1$, for every $\{i, i\oplus1\} \in L_1$, and $\phi\left(\{i, i\oplus2\}\right)=2i$, for every $\{i, i\oplus2\} \in L_2$, is a graph isomorphism. 
\end{proof}

 We also need the following well-known observation.
\begin{proposition}\label{prop-fcycle} 
If $f \in Aut(C_n)$ fixes two adjacent vertices on $C_n$, then $f =id$.
\end{proposition}

Now we present our main result of this section. 
\begin{theorem} Let $n \geq 3$ be an integer. If $n\neq 4$, then $\aut(F_2(C_n))=\aut(C_n)$.
\end{theorem}
\begin{proof}

If $n=3$, then $F_2(C_3)\simeq C_3$ and the result follows for this case. Suppose now that $n\geq 5$. Let $\Gamma=F_2(C_n)$. In the view of Theorem~\ref{aut-sub}, it is enough to show that $|\aut(\Gamma)|\leq 2n$. Let $x \in V(\Gamma)$ be the vertex $\{1, n\}$ and let $\Gamma_1=\Gamma \langle N(x) \rangle$. Since $N(x)=\{\{1, n-1\}, \{2,n\}\}$, then $|\aut(\Gamma_1)|= 2$. Let $\varphi \colon {\rm Stab}(x) \to \aut(\Gamma_1)$ be the function defined by $\varphi(f)=f|_{\Gamma_1}$. As $\varphi$ is a group homomorphism we have that $|{\rm Stab}(x)|\leq |{\rm Ker}~\varphi||\aut(\Gamma_1)|\leq2|{\rm Ker}~\varphi|$. We will prove that ${\rm Ker}~\varphi=\{id\}$. 

Let $f \in {\rm Ker}~\varphi$. By Proposition~\ref{properties2Cn}(\ref{properties2Cn3}) it is enough to show that $f(L_q) \subset  {\rm Fix}(f)$, for every $q \in \{1, \dots, \floor{n/2}\}$. 

First we prove that $f(L_1\cup L_2) \subset {\rm Fix}(f)$. Let $\Gamma_2=\langle L_1 \cup L_2\rangle$. By Proposition~\ref{iso-cycle} it follows that $\Gamma_2\simeq C_{2n}$. Note that $f(L_1\cup L_2)=L_1\cup L_2$. Indeed,  $L_1$ is equal to the set of vertices of degree $2$ in $\Gamma$ and the vertices in $L_2$ are the unique vertices in $\Gamma$ that have two of its neighbors in $L_1$ (see Figure~\ref{fig1} for an example). Then $f|_{L_1\cup L_2} \in \aut(\Gamma_2)$.  Since $f \in {\rm Ker}~\varphi$ and $f \in {\rm Stab}(x)$, we have that $f(\{1, n\})=\{1, n\}$, $f(\{1, n-1\})=\{1, n-1\}$ and $f(\{2, n\})=\{2, n\}$. But $\{1, n\}$ and $\{2, n\}$ are adjacent vertices in $\Gamma$ and then all the vertices in $L_1\cup L_2$ are fixed by $f$ (by Proposition~\ref{prop-fcycle}). If $n=5$ we are done. For $n\geq 6$ we will prove that $f(L_q)\subset {\rm Fix}(f)$, for $3 \leq q \leq  \floor{n/2}$. Suppose by induction that $f(L_j)\subset {\rm Fix}(f)$, for $2\leq j < q \leq \floor{n/2}$. Let $u \in L_q$, that is $u=\{i, i\oplus q\}$, for some $i \in \{1, \dots, n\}$. By Proposition~\ref{properties2Cn}(\ref{properties2Cn4}), we have that two neighbors of $u$, say $v$ and $w$, belongs to $L_{q-1}$ and the vertex $z$ in $N(v)\cap N(w) \setminus \{u\}$ belongs to $L_{q-2}$ (see Figure~\ref{fig2}). By hypothesis, $\{v, w, z\} \subset {\rm Fix}(f)$ which implies that $f(u) \in N(v)\cap N(w)$ and hence $f(u)=u$ (we are using Observation~\ref{obs-vecinos} that shows that $|N(v)\cap N(w)|\leq 2$). Then $f(L_q) \subset {\rm Fix}(f)$ as desired. 

Therefore ${\rm Ker}~\varphi=\{id\}$ and hence $|{\rm Stab}(x)| \leq 2$. Now, since $x \in L_1$, by Proposition~\ref{properties2Cn}(\ref{properties2Cn5}) it follows that $|O(x)| \leq n$. Finally, by the orbit-stabilizer theorem we obtain $|\aut(\Gamma)|=|O(x)||{\rm Stab}(x)| \leq 2n$ which concludes the proof of the theorem. 
\end{proof}
The following conjecture is based on experimental results obtained by computer. 
\begin{conjecture}
Let $n$ be an integer and $3\leq k \leq n/2$. If $k\neq n/2$ then $\aut(F_k(C_n))=\aut(C_n)$ and if $k=n/2$, then $\aut(F_k(C_n))=S_2 \times \aut(C_n)$. 
\end{conjecture}

\section{Automorphism groups of the $2$-token graphs of star, fan and wheel graphs} \label{sec3}
In this section we obtain the automorphism groups of the $2$-token graphs of star, fan and wheel graphs. In all such cases we have that if $|G|\geq 5$, then $\aut(F_2(G))=\aut(G)$.

\subsection*{Star graphs}
First we consider the case of the star graph $K_{1, n-1}$. In Figure~\ref{fig3} we show $F_2(K_{1, 6})$. For small star graphs we have that $F_2(K_{1, 2}) \simeq P_3$ and $F_2(K_{1, 3})\simeq C_6$ and hence $\aut(F_2(K_{1, 2})) \simeq S_2$ and $\aut(F_2(K_{1, 3}))\simeq D_6$. In this section $V(K_{1, n-1})=\{1, 2, \dots, n\}$, where $1$ is the vertex of degree $n-1$. 

\begin{figure}
\begin{center}
\includegraphics[scale=0.7]{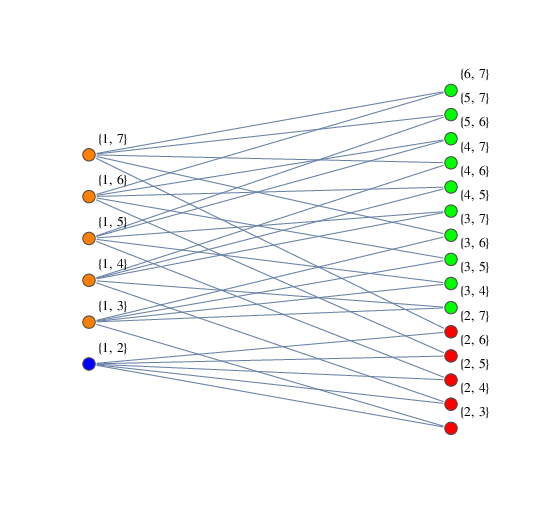}
\caption{The $2$-token graph of $K_{1, 6}$ ($V(K_{1, 6})=\{1, \dots, 7\}$, with $d(1)=7$). The set $B, R, O, G$ defined in the proof of Theorem~\ref{aut-k1n} are the  set of blue, red, orange and green vertices, respectively. }
\label{fig3}
\end{center}
\end{figure}

\begin{theorem}\label{aut-k1n}
Let $n\geq 5$ be an integer. Then 
$\aut(F_2(K_{1,n-1}))= \aut(K_{1, n-1})$.
\end{theorem}
\begin{proof}
Let $\Gamma=F_2(K_{1,n-1})$. Since graph $K_{1, n-1}$ is bipartite, it follows that $\Gamma$ is also bipartite (\cite[Proposition 1]{ruy}). In fact, a bipartition of $V(\Gamma)$ is $\{\mathcal{B}, \mathcal{R}\}$, where 
\[
\mathcal{B}=\{\{1, i\} \colon 2\leq i \leq n  \},\]
 and $ \mathcal{R}=V(\Gamma)) \setminus \mathcal{B}$. Note that if $x \in \mathcal{B}$, then $d(x)=n-2$ and if $x \in  \mathcal{R}$, then $d(x)=2$ (see Figure~\ref{fig3} for an example). Since $n \geq 5$, then $\Gamma$ has exactly $n-1$ vertices of degree $n-2$. Let $B=\{\{1, 2\}\}$ and $O=\{\{1, i\} \colon 3\leq i \leq n\}$. Clearly $\{B, O\}$ is a partition of $\mathcal{B}$. Let $R=\{\{2, i\} \colon 3\leq i \leq n\}$ and $G=\{\{i, j\} \colon 3\leq i <j \leq n\}$. Clearly $R=N(\{1, 2\})$ and $\{R, G\}$ is a partition of $\mathcal{R}$. In Figure~\ref{fig3} we show the partition $\{B, O, R, G\}$ of $V(\Gamma)$.
 
Now, by Theorem~\ref{aut-sub}, $\aut(K_{1, n-1}) \leq \aut(\Gamma)$. It is well-known that $\aut(K_{1, n-1})=S_{n-1}$ and hence it is enough to show that $|\aut(\Gamma)|\leq (n-1)!$ 

Let $x$ be the vertex $\{1, 2\}$ in $\Gamma$ and let $\Gamma_1=\Gamma \langle N(x) \rangle$. In this case $N(x)=R$ and $\Gamma_1=\overline{K_{n-2}}$. Therefore $\aut(\Gamma_1)=S_{n-2}$.  We have that $|O(x)|\leq n-1$ because $x$ has degree $n-2$ and there are exactly $n-1$ vertices in $\Gamma$ of degree $n-2$. If an automorphism $f$ of $\Gamma$ belongs to ${\rm Stab}(x)$, then $f(N(x))=N(x)$ and hence $f\restrict{N(x)} \in \aut(\Gamma_1)$.  Let $\varphi \colon {\rm Stab}(x) \to \aut( \Gamma_1)$ be the function defined by $\varphi(f)=f\restrict{N(x)}$. We will prove that ${\rm Ker}~\varphi=\{id\}$. 

Let $f \in {\rm Ker}~\varphi$. We will prove that $f(Y) \subset{\rm Fix}(f)$, for every $Y \in \{B, O, R, G\}$ in the following order: $B, R, O, G$. Since $f \in {\rm Stab}(x)$ and $f \in {\rm Ker}~\varphi$, then $f(\{1, 2\})=\{1, 2\}$ and $f(y)=y$, for every $y \in N(x)$, respectively. That is, $f(B) \cup f(R) \subset {\rm Fix}(f)$. 
Now we prove that $f(O) \subset  {\rm Fix}(f)$. Let $w \in O$, that is $w=\{1, i\}$, for some $i \in \{3, \dots, n\}$. Note that $N(\{2, i\})=\{\{1, 2\}, \{1, i\}\}$ and hence $\{1, i\} \in N(\{2, i\})$. Since $f(\{2, i\})=\{2, i\}$, then $f(N(\{2, i\}))=N(\{2, i\})$. But $f(\{1, 2\})=\{1, 2\}$ and then $f(\{1, i\})=\{1, i\}$, for every $i \in \{3, \dots, n\}$. Then $f(O) \subset {\rm Fix}(f)$. Finally, let $w  \in G$, that is $w=\{r, s\}$, with $r, s \in \{3, \dots, n\}$, $r\neq s$.  We have that $N(\{r, s\})=\{\{1, r\}, \{1, s\}\}$, that is a subset of $O$. Then $N(\{r, s\}) \subset {\rm Fix}(f)$. Since $|N(y)\cap N(z)|<2$, for every $y, z \in  \mathcal{R}$, with $y \neq z$, then $f(\{r, s\})=\{r, s\}$, for $r, s \in \{3, \dots, n\}$, $r\neq s$. Therefore $f(G) \subset  {\rm Fix}(f)$. We conclude that $f=id$ because $\{B, O, R, G\}$ is a partition of $V(\Gamma)$.  

In this way ${\rm Ker}~\varphi=\{id\}$ and then $|{\rm Stab}(x)| \leq  (n-2)!$. Now we use that $|\aut(\Gamma)|=|O(x)||{\rm Stab}(x)|$ to obtain that $|\aut(\Gamma)|\leq (n-1)(n-2)!=(n-1)!$ as desired. 
\end{proof}

The following conjecture is based on experimental results.
\begin{conjecture}
Let $n$ be an integer and $3\leq k \leq n/2$. If $k\neq n/2$ then $\aut(F_k(K_{1, n}))=\aut(K_{1, n})$, and if $k=n/2$, then $\aut(F_k(K_{1, n}))=S_2 \times \aut(K_{1, n})$. 
\end{conjecture}

\subsection*{Fan graphs}
The fan graph $A_{1, n}$ is the joint graph $K_1+P_{n}$. The vertices of $A_{1, n}$ are the disjoint union $V(K_1) \cup V(P_{n})$, where $V(K_1)=\{v\}$ and $V(P_{n})=\{u_1, \dots, u_n\}$, where $u_i\sim u_{i+1}$ in $P_n$, for every $1\leq i \leq n-1$. In Figure~\ref{fig4} we show $F_2(A_{1, 7})$. 

It is well-known that $\aut(A_{1, n}) \simeq S_2$, for $n \neq 3$. For $n \geq 5$, it can be shown (by Observation~\ref{grado}) that  the degrees of vertices in $F_2(A_{1, n-1})$ are as follows:
\begin{itemize}
\item $d(\{u_1, u_2\})=3, d(\{u_1, u_i\})=5$, for $3\leq i \leq n-1$,  $d(\{u_1, v\})=n$, $d(\{u_1, u_n\})=4$; 
\item $d(\{u_i, u_j\})\in \{4, 6\}$, for $i, j \in \{2, \dots, n-1\}, i \neq j$;
\item $d(\{u_i, u_n\})=5$, for $2\leq i \leq n-2$, $d(\{u_{n-1}, u_{n}\})=3, d(\{u_n, v\})=n$;
\item $d(u_i, v)=n+1$, for every $2\leq i \leq n-1$. 
\end{itemize}

\begin{figure}
\begin{center}
\includegraphics[scale=0.7]{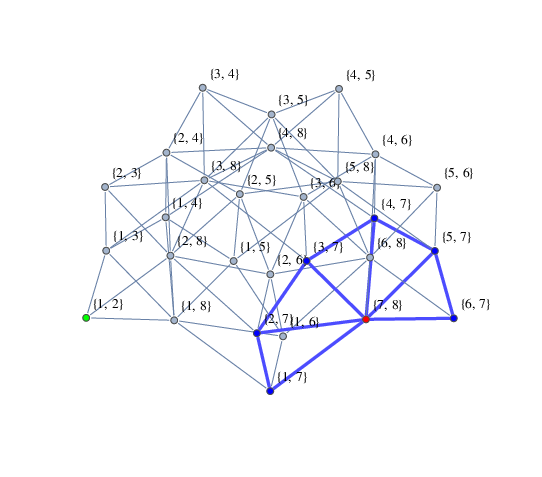}
\caption{{\small The $2$-token graph of $A_{1, 7}$, where $V(A_{1, 7})=\{1, \dots, 8\}$ and $d(8)=7$. In the proof of Theorem~\ref{teo-2fan}, $x=\{1, 2\}$, and $R$ is equal to the set of red and blue vertices.}}
\label{fig4}
\end{center}
\end{figure}

Therefore, if $w \in V(F_2(A_{1, n}))$, then $d(w) \in \{3, 4, 5, 6, n, n+1\}$. Note that for $n\geq 8$, there exists exactly two vertices of degree $3$ and exactly two vertices of degree $n$ in $F_2(A_{1, n})$. We obtained that $\aut(A_{1, 3})\simeq S_2 \times S_2$ by using Mathematica software. 

\begin{theorem}\label{teo-2fan}
Let $n \geq 4$ be an integer. Then $\aut(F_2(A_{1, n}))=\aut(A_{1, n})$.  
\end{theorem}
\begin{proof} 
The proof is by induction on $n$. The cases for $n \in \{4, \dots, 8\}$ were obtained by computer. In the rest of the proof $n \geq 9$. Assume as induction hypothesis that the result is true for every $4\leq m <n$. 

 Let $\Gamma=F_2(A_{1, n})$. By Theorem~\ref{aut-sub}, $\aut(A_{1, n}) \leq \aut(\Gamma)$ and hence it is enough to prove that $|\aut(\Gamma)|\leq 2$. Let
  \[
 R=\left\{\{w, u_n\} \in V(\Gamma) \colon w \in V(A_{1, n}), w \neq u_n\right\}\]
  and $\Gamma_R=\Gamma \langle R \rangle$. We have that 
  \[
  R= N\left(\{v, u_n\}\right) \cup \{\{v, u_n\}\} \setminus \{\{v, u_{n-1}\}\}.
  \]
  Note that $\Gamma_R$ is isomorphic to $A_{1, n-1}$, where $\{v, u_n\}$ is the vertex of degree $n-1$ in $\Gamma_R$. Let  $\Gamma_{\overline{R}}=\Gamma-R$. Note that $\Gamma_{\overline{R}} = F_2(A_{1, n-1})$, where the vertices of $A_{1, n-1}$ are given by $V(K_1)=\{v\}$ and $V(P_{n-1})=\{u_1, \dots, u_{n-1}\}$. 

 Let $x \in V(\Gamma)$ be the vertex $\{u_1, u_2\}$. We have that $|{\rm Orb}(x)|\leq 2$ because $d(x)=3$ and there are exactly two vertices of degree $3$ in $\Gamma$. 

 Let $\Gamma_1=\Gamma\langle N(x)\rangle$, where $N(x)=\{\{u_1, u_3\},  \{v, u_1\}, \{v, u_2\}\}$. Since the graph $\Gamma_1$ is isomorphic to $P_3$, then ${\rm Aut}(\Gamma_1)=S_2$. Let $\varphi \colon {\rm Stab}(x) \to \aut(\Gamma_1)$ be the homomorphism given by $f \mapsto f|_{N(x)}$. Since the vertices $\{u_1, u_3\}$ and $\{v, u_2\}$ have different degree in $\Gamma$, then $f|_{N(x)}=id$, for every $f \in {\rm Stab}(x)$. Therefore ${\rm Im}(\varphi)=\{id\}$ and hence $|{\rm Stab}(x)|=|{\rm Ker}~\varphi|$. We will prove that ${\rm Ker}~\varphi=\{id\}$. 

Let $f \in {\rm Ker}~\varphi$. As $f \in {\rm Stab}(x)$ and $f|_{N(x)}=id$, we have that 
\[
\left\{\{u_1, u_2\}, \{u_1, u_3\}, \{v, u_1\}, \{v, u_2\} \right\} \subseteq {\rm Fix}(f).
\] 
The unique vertices in $\Gamma$ of degree $n$ are $\{v, u_1\}$ and $\{v, u_n\}$. Therefore $f(\{v, u_n\})=\{v, u_n\}$ and this implies that $f(N(\{v, u_n\}))=N(\{v, u_n\})$. Note that $\{v, u_{n-1}\}$ is the unique vertex of degree $n+1$ in $N(\{v, u_n\})$ and hence $f(\{v, u_{n-1}\})=\{v, u_{n-1}\}$. Then $f(R)=R$ and hence $f|_{R} \in \aut(\Gamma_R)$   (remember that $R=N\left(\{v, u_n\}\right) \cup \{\{v, u_n\}\} \setminus \{\{v, u_{n-1}\}\}$). This implies that $f|_{V(\Gamma_{\overline{R}})} \in \aut(\Gamma_{\overline{R}})$. 

Since $\aut(\Gamma_R)$ is isomorphic to $\aut(A_{1, n-1})$, the image of $\{u_1, u_n\}$ under $f|_{R}$ has only two possibilities: $\{u_1, u_n\}$ or $\{u_{n-1}, u_n\}$. But $f(\{u_1, u_n\})=\{u_{n-1}, u_n\}$ is imposible because $\{u_1, u_n\}$ and $\{u_{n-1}, u_n\}$ have different degrees in $\Gamma$. Thus $f|_{R}=id$. Now $\aut(\Gamma_{\overline{R}})=\aut(F_2(A_{1, n-1}))$ and by induction we have that either $f|_{V(\Gamma_{\overline{R}})}=id$ or $f|_{V(\Gamma_{\overline{R}})}=g$, where $g$ is the automorphism in $\aut(F_2(A_{1, n-1}))$ that moves the vertex $\{u_1, u_2\}$. But $f(\{u_1, u_2\})=\{u_1, u_2\}$ and hence $f|_{V(\Gamma_{\overline{R}})}=id$. Therefore $f=id$. In this way $|{\rm Stab}(x)|=1$ and hence $|\aut(\Gamma)| \leq 2$ as desired.

\end{proof}
\begin{conjecture}
Let $n$ be an integer and $3\leq k \leq n/2$. If $k\neq n/2$ then $\aut(F_k(A_{1, n}))=\aut(A_{1, n})$ and if $k=n/2$, then $\aut(F_k(A_{1, n}))=S_2 \times \aut(A_{1, n})$. 
\end{conjecture}

\subsection*{Wheel graphs}

The wheel graph $W_{1, n}$, $n\geq 3$, is defined as the join graph $K_1+C_n$. In Figure~\ref{fig5} we show $F_2(W_{1, 7})$. It is well-known that $\aut(W_{1, n})\simeq D_{2n}$. We obtained that $\aut(W_{1, 3})\simeq S_2\times D_{6}$ by using Mathematica software. 

\begin{figure}
\begin{center}
\includegraphics[scale=0.7]{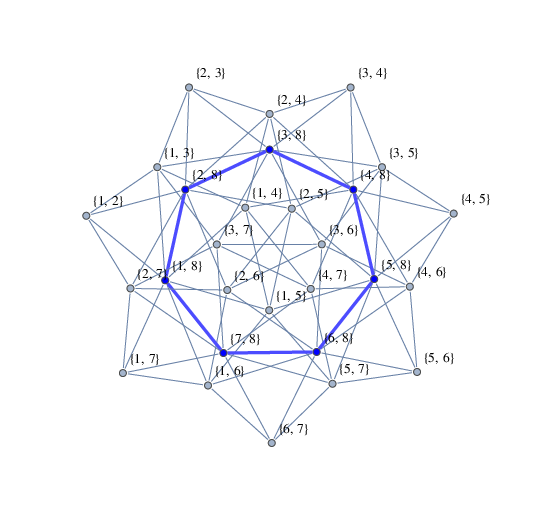}
\caption{{\small The $2$-token graph of $W_{1, 7}$, where $V(W_{1, 7})=\{1, \dots, 8\}$ and $d(8)=7$. In the proof of Theorem~\ref{teo-2rueda}, $\Gamma_C$ is the blue subgraph.}}
\label{fig5}
\end{center}
\end{figure}

\begin{theorem}\label{teo-2rueda} Let $n \geq 4$ be integer. Then $\aut(F_2(W_{1, n}))=\aut(W_{1, n})$.
\end{theorem}
\begin{proof} 
In this proof, the vertex set of $W_{1, n}$ is $\{v, u_1, \dots, u_n\}$, where $V(K_1)=\{v\}$ and $V(C_n)=\{u_1, \dots, u_n\}$, with $u_1\sim u_n$ and $u_i\sim u_{i+1}$, $1\leq i \leq n$ in $C_n$. The cases $n \in \{4, 5\}$ were verified by computer and hence we suppose that $n \geq 6$. Let $\Gamma=F_2(W_{1, n})$. Let $T=\{\{u_i, u_j\} \in V(\Gamma) \colon 1\leq i < j \leq n\}$ and $C=\{\{u_i, v\} \in V(\Gamma)  \colon 1 \leq i \leq n\}$. Let $\Gamma_T=\Gamma\langle T \rangle$ and $\Gamma_C=\Gamma\langle C \rangle$. We have that $\Gamma_T = F_2(C_n)$ and $\Gamma_C\simeq C_n$ (see Figure~\ref{fig5} for an example), where an isomorphism between $C_n$ and $\Gamma_C$ is given by $u_i \mapsto (u_i, v)$. In this proof we use that every automorphism in $\aut(\Gamma_T)$ (that is equal to $\aut(F_2(C_n))$) is induced by some automorphism in $\aut(C_n)$. For the case of $\aut(\Gamma_C)$ we have that every automorphism $\theta \in \aut(C_n)$ induces an automorphism $g$ in $\aut(\Gamma_C)$ given by $g(\{u_i, v\})=(\{\theta(u_i), v\})$. We know that $\aut(W_{1, n}) \leq \aut(F_2(W_{1, n}) )$ and we will prove that $\aut(F_2(W_{1, n}) )=\aut(W_{1, n})$ by showing that every automorphism $f$ in  $\aut(F_2(W_{1, n}))$ is induced by some automorphism $\theta$ in $\aut(W_{1, n})$, i.e., $f(\{a, b\})=\{\theta(a), \theta(b)\}$, for every $\{a, b\} \in V(F_2(W_{1, n}))$. 

Let $y \in V(F_2(W_{1, n}))$. Note that if $y \in T$, then $d(y) \in \{4, 6\}$, and if $ y\in C$, then $d(y)=n+1$. As $n \geq 6$, then $d(y) \not \in \{4, 6\}$ when $y \in C$. Therefore, we have that $f|_T\in \aut(\Gamma_T)$ and $f|_C \in \aut(\Gamma_C)$, for every $f \in \aut(F_2(W_{1, n}))$. As $f|_T \in \aut(F_2(C_n))$, then there exists $\alpha \in \aut(C_n)$ such that $f|_T(\{a, b\})=\{\alpha(a), \alpha(b)\}$, for any $\{a, b\} \in V(\Gamma_T)$. For the case of $f|_C$ we have that $f|_C(\{a, v\})=\{\beta(a), v\}$, for some $\beta \in \aut(C_n)$. 

We will prove by contradiction that $\alpha=\beta$. Suppose that $\alpha\neq \beta$. Without loss of generality, we suppose that $\alpha(u_1)\neq \beta(u_1)$. 
\begin{claim}
 If $\alpha(u_1)\neq \beta(u_1)$, then $\alpha(u_1)=\beta(u_2), \alpha(u_2)=\beta(u_1)$ and $\alpha(u_j)=\beta(u_j)$, for every $j \in \{3, \dots, n\}$.
 \end{claim}
 \begin{proof} 
As $\{u_1, u_2\} \sim \{u_1, v\}$, then $f(\{u_1, u_2\}) \sim f(\{u_1, v\})$, that is $\{\alpha(u_1), \alpha(u_2)\}\sim \{\beta(u_1), v\}$. By the definition of $2$-token graph $|\{\alpha(u_1), \alpha(u_2)\}\cap \{\beta(u_1), v\}|=1$. But $v \not \in \{\alpha(u_1), \alpha(u_2)\}$ and $\alpha(u_1) \neq \beta (u_1)$, and then we have that $\alpha(u_2)=\beta(u_1)$. On the other hand $\{u_1, u_2\} \sim \{u_2, v\}$ and then $f(\{u_1, u_2\}) \sim f(\{u_2, v\})$. That is $\{\alpha(u_1), \alpha(u_2)\}\sim \{\beta(u_2), v\}$. But we have proved that $\alpha(u_2)=\beta(u_1)$ and hence  $\{\alpha(u_1), \beta(u_1)\} \sim \{\beta(u_2), v\}$. Similarly as in previous case the equality $|\{\alpha(u_1), \beta(u_1)\}\cap \{\beta(u_2), v\}|=1$ implies that $\alpha(u_1)=\beta(u_2)$. 

 Now, for $j \in \{3, \dots, n\}$, we have that $\{u_1, u_j\}\sim \{u_j, v\}$. Then 
 $f(\{u_1, u_j\})\sim f(\{u_j, v\})$, that is $\{\alpha(u_1), \alpha(u_j)\}\sim \{\beta(u_j), v\}$. But $\alpha(u_1)=\beta(u_2)$ and hence $\{\beta(u_2), \alpha(u_j)\}\sim \{\beta(u_j), v\}$. 
 The unique option is that $\alpha(u_j)=\beta(u_j)$ and the proof of the claim is completed.
\end{proof}
 Let $a=\beta(u_1)$ and $b=\beta(u_2)$. Using previous claim, it is easy to check that $\alpha \beta^{-1}(a)=b$, $\alpha \beta^{-1}(b)=a$, and $\alpha \beta^{-1}(c)=c$, for every $c \in V(C_n) \setminus \{a, b\}$. That is, $\alpha \beta^{-1}$ is equal to the transposition $(a, b)$ (written in cyclic notation). 
 But $\alpha \beta^{-1} \in \aut(C_n)=D_{2n}$ and the dihedral group $D_{2n}$ has not transpositions when $n >3$. 
 This contradiction shows that $\alpha=\beta$. Therefore, $f$ is the automorphism induced by $\theta \in \aut(W_{1, n})$, where $\theta(v)=v$ and $\theta(u_i)=\alpha(u_i)$, for every $i \in \{1, \dots, n\}$.

 \end{proof}
 \begin{conjecture}
Let $n$ be an integer and $3\leq k \leq n/2$. If $k\neq n/2$ then $\aut(F_k(W_{1, n}))=\aut(W_{1, n})$ and if $k=n/2$, then $\aut(F_k(W_{1, n}))=S_2 \times \aut(W_{1, n})$. 
\end{conjecture}

\section{Proof of Theorem~\ref{auto-path}}\label{sec4}

In this section we use $P_n$ to denote the path graph with $V(P_n)=\{1, \dots, n\}$ and $E(P_n)=\{\{i, i+1\}\colon 1\leq i \leq n-1 \}$. It is well-known that $\aut(P_1)\simeq S_1$ and $\aut(P_n) \simeq S_2$, for $n \geq 2$. Explicitly, if $n\geq 2$, then $\aut(P_n)=\{id, \theta\}$, where 
\[
\theta=(1, \; n)(2 ,\; n-1) \dots (n/2, \; n/2+1),  \text{ when $n$ is even}
\] 
and 
\[
\theta=(1, \; n)(2, \; n-1) \dots (\lceil n/2 \rceil-1, \; \lceil n/2 \rceil+1)(\lceil n/2 \rceil),   \text{ when $n$ is odd}.
\]
(We are writing permutation $\theta$ in its cycle notation).

Our main result in this section is Theorem~\ref{auto-path} about the automorphism group of $F_k(P_n)$, for $n \neq 2k$. In Figure~\ref{fig6} and \ref{fig7} we show $F_2(P_6)$ and $F_3(P_7)$, respectively. 

\begin{figure}
\begin{center}
\includegraphics[width=.70\textwidth]{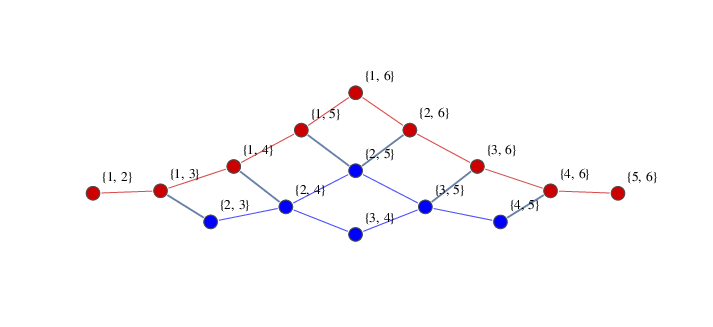}
\caption{The $2$-token graph of $P_6$. The red (resp. blue) subgraph is induced by  the set $L_1$ (resp. $L_2$) defined in the proof of Theorem~\ref{auto-path}.}
\label{fig6}
\end{center}
\end{figure}

First, we present some auxiliary results. Without loss of generality, we write the elements of a vertex  $\{a_1, \dots, a_k\} \in V(F_k(P_n))$ in ascending order, that is $a_1<\dots <a_k$. Using Observation~\ref{grado} we obtain the following facts about the vertices in $F_k(P_n)$.
 \begin{observation}\label{ob-p-2} Let $P_n$ be a path graph of order $n \geq 3$ and let  $2\leq k \leq n-2$ be an integer. 
 \begin{enumerate}  \item \label{ob-p-2-1}  Let $v=\{a_1, a_2, \dots, a_k\}$ be a vertex in $F_k(P_n)$. The degree of $v$ is even if and only if  $\{a_1, a_k\}=\{1, n\}$ or $\{a_1, a_k\}\cap \{1, n\}=\emptyset$.
\item \label{ob-p-2-2} The vertices of degree $2$ in $F_k(P_n)$ are either of the form $\{a, a+1, \dots, a+(k-1)\}$, for $2 \leq a \leq n-k$, or $\{1, \dots, m, n-(k-m-1), n-(k-m), \dots, n\}$, for $1\leq m\leq k-1$.
\item \label{ob-p-2-3} The graph $F_k(P_n)$ has exactly two vertices of degree one, say $\{1, \dots, k\}$ and $\{n-(k-1), \dots, n\}$.
\end{enumerate}
\end{observation}

The formula for the graph distance  $d(u, v)$ between vertices $u$ and $v$ in $F_k(P_n)$ is given in the following result. 

\begin{lemma}
Let $u=\{u_1, \dots, u_k\}$ and $ v=\{v_1, \dots, v_k\}$ be vertices in $F_k(P_n)$, where $u_1 <  \dots < u_k$ and $v_1 < \dots <  v_k$. Then
\[
d( u,  v)=\sum_{i=1}^k|v_i-u_i|.
\]
\end{lemma}
\begin{proof}
The $L_1$-distance on $\zz^k$ is defined as 
\[
d_{L_1}( a,  b)=\sum_{i=1}^k|b_i-a_i|,
\]
for every $ a=\{a_1, \dots, a_k\},  b=\{b_1, \dots, b_k\} \in \zz^k$. The grid graph $\zz^k$ is constructed as follows: two points in $\zz^k$ are adjacent if their $L_1$-distance is equal to one.  It is well-known (see, e.g., \cite[p. 333]{deza}) that the $L_1$-distance is equal to the path distance on the grid graph $\zz^k$. The result follows by observing that $F_k(P_n)$ is a subgraph of the grid graph $\zz^k$, where vertex $\{a_{i_1}, \dots, a_{i_k}\}$ in $F_k(P_n)$ correspond to the unique vertex $(a_1, \dots, a_k)$ in $\zz^k$ such that $a_1 < \dots < a_k$ and $\{a_{i_1}, \dots, a_{i_k}\}=\{a_1, \dots, a_k\}$.
\end{proof}

The proof of previous lemma also follows by using the formula of distance in \cite{fisch}. The case for $F_2(P_n)$ was also proved in \cite{beaula}. In the proof of the following theorem, we use several times the following fact:  if $H$ is a graph, then
\[
d(a, b)=d\left(g(a), g(b)\right),
\] for every $g \in \aut(H)$ and every $a, b \in V(H)$.

\begin{proof}[Proof of Theorem~\ref{auto-path}]

 If $k=1$, then $F_1(P_n) \simeq P_n$ and the result is trivially true. The cases $n\leq 7$, $1\leq k \leq n/2$, were verified by computer and hence we suppose that $n \geq 8$.  Let $\Gamma=F_k(P_n)$.  
By Theorem~\ref{aut-sub}, it follows that $\aut(P_n) \leq \aut(\Gamma)$ and hence it is enough to prove that $|\aut(\Gamma)|\leq2$.

Let $x \in \Gamma$ be the vertex $\{1, \dots, k\}$. Note that $d(x)=1$ and hence $|O(x)|\leq 2$. In fact, $|O(x)|= 2$ because the non-identity automorphism, say $\phi$, in $\aut(P_n)$ induces a non-identity automorphism $g_\phi$ in $\aut(\Gamma)$  such that $g_\phi(\{1, \dots, k\})=\{n-(k-1), \dots, n\}$. Let $\Gamma_1=\Gamma\langle N(x) \rangle$.  Let $\varphi \colon {\rm Stab}(x) \to \aut( \Gamma_1)$ be the function defined by $\varphi(f)=f\restrict{N(x)}$. As $\aut(\Gamma_1)=\{ id\}$, we have that ${\rm Stab}(x)={\rm Ker}~\varphi$. We will prove that ${\rm Ker}~\varphi=\{ id\}$, which implies that $|{\rm Aut}(\Gamma)| \leq 2$. 

First we prove the case $k=2$. Suppose by induction that $F_2(P_m)=\aut(P_m)$, for every $5\leq m <n$. Let $L_1=\{\{a_1, a_2\} \in V(\Gamma) \colon a_1=1 \text{ or } a_2=n\}$ and $L_2=V(\Gamma) \setminus L_1$. Let $\Gamma_{L_1}=\Gamma \langle L_1 \rangle$ and let $\Gamma_{L_2}=\Gamma \langle L_2 \rangle$. Note that $\Gamma_{L_1}\simeq P_{2n-3}$ and, by Proposition~\ref{pborrado}, we have that $\Gamma_{L_2}\simeq F_2(P_{n-2})$ (see Figure~\ref{fig6} for an example). It is easy to see that, with the exception of $\{1, n\}$, all the vertices in $L_1$ have either degree $1$ or degree $3$. In fact $L_1$ has all the vertices of $\Gamma$ that have degree $1$ or $3$. Note that $\{1, n\}$ is the unique vertex in $\Gamma$ of degree $2$ in  with its two neighbors of  degree $3$. Consequently, if $f \in \aut(\Gamma)$, then $f(L_1)=L_1$ and hence $f|_{L_1} \in \aut(\Gamma_{L_1})$. Now, let $f \in {\rm Ker}~\varphi$. Since $\aut(L_1)\simeq \aut(P_{2n-3})$ and  $f(\{1, 2\})=\{1, 2\}$, we have that $f|_{L_1}=id$ which shows that $f(L_1) \subset  {\rm Fix}(f)$. 
Using this, the proof  that $f(L_2) \subset {\rm Fix}(f)$ follows immediately by induction because $\Gamma_{L_2}\simeq F_2(P_{n-2})$ and $n-2\geq 6$ (see Figure~\ref{fig6} for an example). 

Now we prove the case $3 \leq k <n/2$. The proof  is by induction on $n$. As bases cases, we have proved the result for $F_2(P_n)$, for every $n\geq 5$, and for $n=8$, we have verified every $k \in \{2, 3\}$ by computer. We suppose that the result is true for any graph $P_{n'}$ with $n'<n$. That is, $\aut(F_k'(P_{n'}))=\aut(P_{n'})$, for $2\leq k' <  n'/2<n/2$. 

Let $f \in {\rm Ker}~\varphi$. Since $f \in {\rm Stab}(x)$, we have that $f(\{1, \dots, k-1, k+1\})=(\{1, \dots, k-1, k+1\}$ and $f(\{n-(k-1), \dots, n-1, n\})=(\{n-(k-1), \dots, n-1, n\})$ (the grey vertices in the example in Figure~\ref{fig7}). 

We define the following subsets of $V(\Gamma)$
\begin{eqnarray*}
A&=&\{\{a_1, \dots, a_k\} \in V(\Gamma) \colon a_k=n\}\\
B&=&\{\{b_1, \dots, b_k\} \in V(\Gamma) \colon b_1=1\}\\
C&=&\{\{c_1, \dots, c_k\}  \in V(\Gamma) \colon c_1\neq 1, c_k\neq n\}
\end{eqnarray*} 
Note that $C=V(\Gamma) \setminus (A\cup B)$. Let $\Gamma_A=\Gamma\langle A \rangle$, $\Gamma_B=\Gamma\langle B \rangle$ and $\Gamma_C=\Gamma\langle C \rangle$. We will prove that if $f$ in ${\rm Ker}~\varphi$, then $f$ is the identity permutation by showing that $X \subset  {\rm Fix}(f)$, for every $X \in \{A, B, C\}$. We do this in four steps: first we prove that if $f \in {\rm Ker}~\varphi$, then $f(\{1, \dots, k-1, n\})=(\{1, \dots, k-1, n\}$ (Claim~\ref{clamita2}), second we work the case of $A$, third the case of $B$ and finally the case of $C$.   In Figure~\ref{fig7}, we use colors to illustrate the four steps.

\begin{figure}
\begin{center}
\includegraphics[width=.9\textwidth]{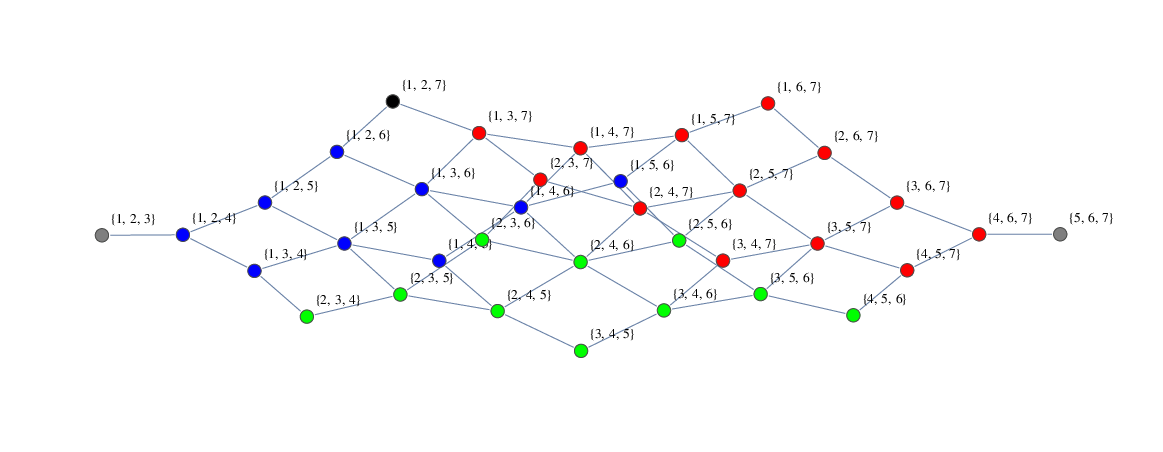}
\caption{The $3$-token graph of $P_7$. In the proof of Theorem~\ref{auto-path}, for $f \in {\rm Ker}~\varphi$, we prove that $f(v)=v$, for every $v \in V(\Gamma)$, in the following order: grey, black, red, blue and green vertices.}\label{fig7}
\end{center}
\end{figure}

We need the following:

\begin{claim}\label{claim1} Let $y$ be the vertex $\{1, \dots, k-1, n\}$.
\begin{enumerate}
\item If $3 \leq a \leq n-k-1$, then $\{a, a+1, \dots, a+(k-1)\} \not \in O(y)$,
\item If $2\leq m\leq k-2$, then $\{1, \dots, m, n-(k-m-1), n-(k-m), \dots, n\} \not \in O(y)$.
\end{enumerate}
\end{claim}
\begin{proof}
The elements in $O(y)$ should be vertices of degree $2$. We will use the fact that if $g \in \aut(\Gamma)$, then $N(g(v))=g(N(v))$, for every $v \in V(\Gamma)$. Note that $N(y)=\{y_1, y_2\}$, where $y_1=\{1, \dots, k-2, k, n\}$ and $y_2=\{1, \dots, k-1, n-1\}$. The condition $k\geq 3$ implies that $d(y_1)=4$ and $d(y_2)=3$.

{\it Proof of (1)}. Let $w=\{a, a+1, \dots, a+(k-1)\}$, with  $3 \leq a \leq n-k-1$, then $N(w)=\{w_1, w_2\}$, where $w_1=\{a-1, a+1, \dots, a+(k-1) \}$, $w_2=\{a, a+1, \dots, a+(k-2), a+k\}$, $d(w_1)=d(w_2)=4$. Therefore $f(y)\neq w$.

{\it Proof of (2)}. Let $z=\{1, \dots, m, n-(k-m-1), n-(k-m), \dots, n\}$, with $2\leq m\leq k-2$. Then $N(z)=\{z_1, z_2\}$, where $z_1=\{\{1, \dots, m-1, m+1, n-(k-m-1), n-(k-m), \dots, n\}$, $z_2=\{1, \dots, m, n-(k-m-1)-1, n-(k-m), \dots, n\}$ and $d(z_1)=d(z_2)=4$ (here we are using that $k <n/2$). Therefore $f(y)\neq z$.
\end{proof}

{\bf Step 1}. We prove the following fact.  
\begin{claim}\label{clamita2} Let $y$ be the vertex $\{1, \dots, k-1, n\}$. If $f \in {\rm Ker}~\varphi$, then $f(y)=y$.
\end{claim}
\begin{proof}
By Observation~\ref{ob-p-2}(\ref{ob-p-2-2}) and Claim~\ref{claim1} we have that $f(y)$ has only the following options: 
\begin{enumerate}
\item $ \{2, 3, \dots, k+1\},$ 
\item $\{n-k, n-(k-1), \dots, n-1\},$
\item $\{1, n-(k-2), \dots, n\}$
\item $\{1, \dots, k-1, n\}$,
\end{enumerate}
Now we use several times the fact that $d(a, b)=d\left(g(a), g(b)\right)$, for every $g \in \aut(\Gamma)$ and every $a, b \in V(\Gamma)$.

{\it Case 1.} Suppose that $f(\{1, \dots, k-1, n\})= \{2, 3, \dots, k+1\}$. Then
\begin{eqnarray*}
d(\{1, 2, \dots, k\}, \{1, \dots, k-1, n\})&=&d(\{1, 2, \dots, k\}, \{2, 3, \dots, k+1\})\\
n-k&=&k
\end{eqnarray*}
which implies that $n=2k$, a contradiction.

{\it Case 2.} Suppose that $f(\{1, \dots, k-1, n\})= \{n-k, n-(k-1), \dots, n-1\}$. Then
\begin{eqnarray*}
d(\{1, 2, \dots, k\}, \{1, \dots, k-1, n\})&=&d(\{1, 2, \dots, k\}, \{n-k, n-(k-1), \dots, n-1\})\\
n-k&=&\sum_{i=1}^{k}(n-i)-\sum_{i=1}^ki
\end{eqnarray*}
From which we obtain $k^2-nk+n=0$. The discriminant of equation $x^2-nx+n=0$ is $D=n^2-4n$. As $n\geq 7$, $D$ is positive and not a square number. Therefore $x^2-nx+n=0$ has not integer solutions, which is a contradiction.

{\it Case 3.} Suppose that $f(\{1, \dots, k-1, n\})= \{1, n-(k-2), \dots, n\}\}$. Then 
\begin{eqnarray*}
d(\{1, 2, \dots, k\}, \{1, \dots, k-1, n\})&=&d(\{1, 2, \dots, k\},  \{1, n-(k-2), \dots, n\})\\
n-k&=&\sum_{i=0}^{k-2}(n-i)-\sum_{i=2}^ki\\
n-k&=&(n-k)(k-1)\\
\end{eqnarray*}
As $n \neq k$, then $k=2$, a contradiction that $k\geq 3$. 

Therefore, the unique option is that $f(y)=y$ as desired. 
\end{proof}
In Figure~\ref{fig7} we show an example where vertex $\{1, \dots, k-1, n\}$ is colored black. 

 {\bf Step 2.} We will prove that $f(A) \subset  {\rm Fix}(f)$ (the red vertices in the example in Figure~\ref{fig7}). 
 
 First we prove that $f(A)=A$. Let $v=\{a_1, \dots, a_{k-1}, n\}$ be a vertex in $A$. Suppose that $f(v)=\{c_1, \dots, c_k\}$. Then
\[
d(\{1, 2, \dots, k\}, \{a_1, \dots, a_{k-1}, n\})=d(\{1, 2, \dots, k\}, \{c_1, \dots, c_k\}),
\]
 from which follows that 
\begin{equation}\label{eq-p-1}
n+a_1+\dots +a_{k-1}=c_1+c_2+\dots +c_k
\end{equation}

 Now, by Claim~\ref{clamita2}, $f(\{1, \dots, k-1, n\})=\{1, \dots, k-1, n\}$. Then 
\[
d(\{1, \dots, k-1, n\}, \{a_1, \dots, a_{k-1}, n\})=d(\{1, \dots, k-1, n\}, \{c_1, \dots, c_k\}),
\]
 from which follows that
\begin{equation}\label{eq-p-2}
a_1+\dots +a_{k-1}=n-c_k+(c_1+c_2+\dots +c_{k-1})
\end{equation}

 Combining equations (\ref{eq-p-1}) and (\ref{eq-p-2}) we obtain that $c_k=n$ and hence $f(v) \in A$. Then $f(A) \subseteq A$. In fact, $f(A) = A$ because $f$ is a bijection. Therefore $f|_A \in \aut(\Gamma_A)$ .  

 Notice that $\Gamma_A$ is isomorphic to $F_{k-1}(P_{n-1})$ (see Figure~\ref{fig8} for the case of $F_3(P_7)$), with an  isomorphism given by $\{a_1, \dots, a_{k-1}, n\} \mapsto \{a_1, \dots, a_{k-1}\}$. The inequality $3\leq k< n/2$, implies that $2\leq k-1< (n-1)/2$.  By the induction hypothesis, $\aut(F_{k-1}(P_{n-1}))=S_2$. The unique option is that $f|_A=id$ because $\{n-(k-1), \dots, n-1, n\}$ is the unique vertex in $A$ with degree $1$ simultaneously in $\Gamma_A$ and $F_k(P_n)$ (see Figure~\ref{fig8} for an example).
 
 \begin{figure}
\begin{center}
\includegraphics[width=.85\textwidth]{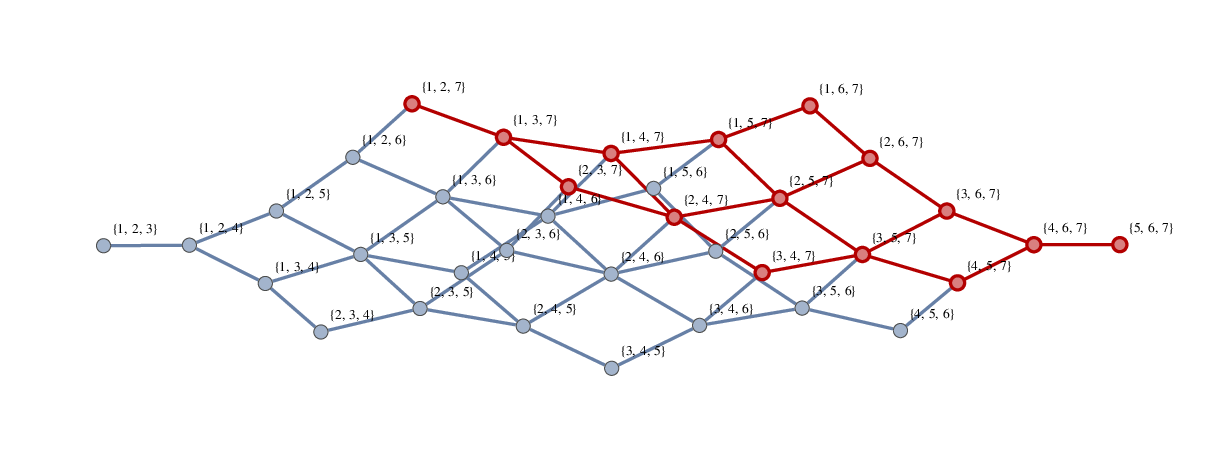}
\caption{The $3$-token graph of $P_7$. The red subgraph $H$ is isomorphic to $F_2(P_6)$. Note that the vertex $\{1, 2, 7\}$ has degree one in $H$ but degree two in $F_3(P_7)$.}
\label{fig8}
\end{center}
\end{figure}

 {\bf Step 3.} We will prove that $f(B) \subset  {\rm Fix}(f)$ (the blue vertices in the example in Figure~\ref{fig7}). 
 
 First, we prove that $f(B)=B$. Let $v=\{1, v_2, \dots, v_k\} \in B$. We have two cases: $v_k=n$ or $v_k \neq n$. If $v_k=n$, then $v \in A$ and hence $f(v)\in B$, because we have proved that $f(v)=v$, for every $v \in A$. Suppose now that $v_k\neq n$, that is $v \not \in A$. By Observation~\ref{ob-p-2}(\ref{ob-p-2-1}) it follows that $d(v)$ is odd and hence $d(f(v))$ should be odd. Also by Observation~\ref{ob-p-2}(\ref{ob-p-2-1}), $f(v)$ has only two options. The first one is $f(v)=\{1, \dots, d_k\}$, with $d_k \neq n$, in which case  $f(v)\in B$ as desired. The second option is $f(v)=\{a_1, \dots, n\}$, with $a_1\neq 1$. In this case $f(v) \in A$ and hence $f(v)=v \in A$, a contradiction. 

 Therefore, $f(B)\subseteq B$ and hence $f(B)=B$ because $f$ is a bijection. As $B\simeq F_{k-1}(P_{n-1})$, with an isomorphism given by $\{1, b_2, \dots, b_k\}\mapsto \{b_2, \dots, b_k\}$, then $\aut(\Gamma_B)\simeq S_2$ and hence $f|_B=id$ as in the case of $\Gamma_A$. Therefore $f(A\cup B)\subset {\rm Fix}(f)$. 

 {\bf Step 4.} We will prove that $f(C) \subset  {\rm Fix}(f)$ (the green vertices in the example in Figure~\ref{fig7}).
 
It is easy to see that $\Gamma \langle C \rangle \simeq F_k(P_{n-2})$. However, we can not apply the induction hypothesis on $C$ because it is possible that $k \geq (n-2)/2$ (for example, if $n=10$ and $k=4$). We solve this inconvenient by using the graph distance in $\Gamma$.   

Suppose that $f(\{c_1, \dots, c_k\})=\{d_1, \dots, d_k\}$.
\begin{eqnarray*}
d(\{1,c_2, \dots, c_k\}, \{c_1, c_2, \dots, c_k\})&=&d(\{1,c_2, \dots, c_k\}, \{d_1, \dots, d_k\})\\
c_1-1&=&d_1-1+\sum_{i=2}^k|d_i-c_i|
\end{eqnarray*}
and this implies that $c_1\geq d_1$. Now
\begin{eqnarray*}
d(\{1,d_2, \dots, d_k\}, \{c_1, c_2, \dots, c_k\})&=&d(\{1,d_2, \dots, d_k\}, \{d_1, \dots, d_k\})\\
c_1-1+\sum_{i=2}^k|d_i-c_i|
 &=&d_1-1 
 \end{eqnarray*}
 and this implies that $d_1\geq c_1$. Therefore, $d_1=c_1$.
 Now we will prove that $d_k=c_k$.
 \begin{eqnarray*}
d(\{c_1,c_2, \dots, c_{k-1},n\}, \{c_1, \dots, c_k\})&=&d(\{c_1,c_2, \dots, c_{k-1}, n\}, \{d_1, \dots, d_k\})\\
n-c_k&=&n-d_k+\sum_{i=1}^{k-1}|d_i-c_i|
\end{eqnarray*}
and this implies that $d_k\geq c_k$. Now
\begin{eqnarray*}
d(\{d_1,d_2, \dots, d_{k-1}, n\}, \{c_1, c_2, \dots, c_k\})&=&d(\{d_1,d_2, \dots, d_{k-1}, n\}, \{d_1, \dots, d_k\})\\
n-c_k+\sum_{i=1}^{k-1}|d_i-c_i|
 &=&n-d_k 
 \end{eqnarray*}
 and this implies that $d_k\leq c_k$. Therefore, $d_k=c_k$.
 
 Finally we prove that $c_i=d_i$, for every $2\leq i \leq k-1$. 
We have proved that if $c=\{c_1,c_2, \dots, c_{k-1},c_k\}$, then $f(c)=\{c_1, d_2, \dots, d_{k-1}, c_k\}$. Therefore
 \begin{eqnarray*}
d(\{1,d_2, \dots, d_{k-1}, n\}, \{c_1, c_2, \dots, c_{k-1}, c_k\})&=&d(\{1, d_2, \dots, d_{k-1}, n\}, \{c_1,d_2, \dots, d_{k-1},c_k\})\\
n-c_k+c_1-1+\sum_{i=2}^{k-1}|d_i-c_i|
 &=&n-c_k+c_1-1 
 \end{eqnarray*}
 From which we obtain that  $\sum_{i=2}^{k-1}|d_i-c_i|=0$. But as every $|d_i-c_i|$ is non negative, then $|d_i-c_i|=0$, for every $i \in \{2, \dots, k-1\}$ and hence $d_i=c_i$, for every $i \in \{2, \dots, k-1\}$. 
 So we have that $f(c)=c$, for every $c \in C$.
 
Therefore $f=id$ and the proof of the theorem is completed. 
\end{proof}
We finish this paper with the following: 
\begin{conjecture}
Let $n$ be an even integer. Then $\aut(F_{n/2}(P_n))\simeq S_2 \times S_2$.
\end{conjecture}

\section*{Acknowledgments}

The authors would like to thank the reviewer for her/his useful corrections.

\end{document}